\documentclass[12pt]{amsart}
\usepackage{amscd,amssymb,epsfig,mathtools}
\usepackage[usenames,dvipsnames]{color}

\usepackage{algpseudocode}
\usepackage{algorithm}

\calclayout

\DeclareMathOperator*{\argmin}{\arg\!\min} 
 
\numberwithin{equation}{section}

\newcommand{\R}{\mathbb{R}}

\newcommand\Conv{\operatorname{Conv}}

\newcommand\diam{\operatorname{diam}}

\newcommand{\tir}[1]{\ensuremath{\overline {#1}}}

\newtheorem{thm}{Theorem}[section] 
 
\newtheorem{lemma}[thm]{Lemma}

\newtheorem{defn}[thm]{Definition} 
 
\newtheorem{rem}[thm]{Remark}

\newtheorem{ass}[thm]{Assumption}

\def\whsq{\vbox to 5.8pt 
{\offinterlineskip\hrule 
\hbox to 5.8pt{\vrule height 
5.1pt\hss\vrule height 5.1pt}\hrule}}

\def\<{\langle} 
\def\>{\rangle} 
\def\PP{{\mathop{{\rm I}\kern-.2em{\rm P}}\nolimits}} 
\def\FF{{\mathop{{\rm I}\kern-.2em{\rm F}}\nolimits}}   
\def\ZZ{{\mathop{{\rm I}\kern-.2em{\rm Z}}\nolimits}} 
 
\newlength{\sidemargin} 
\begin{document}

\algnewcommand{\algorithmicgoto}{\textbf{go to}}%
\algnewcommand{\Goto}[1]{\algorithmicgoto~\ref{#1}}%

\title[]{
Convergence of a damped Newton's method for discrete Monge-Amp\`ere functions with a prescribed asymptotic cone}

\thanks{ }
\author{Gerard Awanou}
\address{Department of Mathematics, Statistics, and Computer Science, M/C 249.
University of Illinois at Chicago, 
Chicago, IL 60607-7045, USA}
\email{awanou@uic.edu}  
\urladdr{http://www.math.uic.edu/\~{}awanou}

\maketitle

\begin{abstract} 

For finite difference discretizations with linear complexity and  provably convergent to weak solutions of the second boundary value problem for the Monge-Amp\`ere equation, we give the first proof of uniqueness. The boundary condition is enforced through the use of the notion of asymptotic cone while the differential operator is discretized based on a discrete analogue of the subdifferential.
We establish the convergence of a subsequence of a damped Newton's method for the nonlinear system resulting from the discretization, thereby proving the existence of a solution. Using related arguments we then prove that such a solution is necessarily unique. Convergence of the discretization as well as numerical experiments are given.

\end{abstract}

\section{Introduction}

We prove the uniqueness of solutions to a finite difference approximation of the second boundary value problem for Monge-Amp\`ere type equations. The discretization we consider was previously used for numerical experiments in  \cite{Awanou-second-sym} where a convergence analysis of a variant of the method considered in this paper was also given. 

We establish the existence of a solution to the discretization as the limit of the iterates of a damped Newton's method for solving the nonlinear system resulting from the discretization. We then observe that such a solution must be unique. In addition, we prove convergence of the discrete solutions to weak solutions of the continuous problem. 

Monge-Amp\`ere type equations with the second boundary value condition arise in geometric optics and optimal transport. 
The approach in \cite{Awanou-second-sym}, as well as the one in this paper, interprets the boundary condition as prescribing the asymptotic cone of the epigraph of the convex solution to the Monge-Amp\`ere equation. 

\subsection{Weak solutions of the Monge-Amp\`ere equation}

Let $\Omega$ be a bounded convex domain of $\R^d$, and let 
$\Omega^*$ be a bounded convex polygonal domain of $\R^d, d \geq 1$. We consider a locally integrable function $R$ on $\R^d$ such that $R>0$ on $\Omega^*$ and $R=0$ on $\R^d \setminus \Omega^*$.  Additionally, let $f \geq 0$ be an integrable function on $\Omega$, and assume that the compatibility condition 
\begin{equation} \label{necessary}
\int_{\Omega} f(x) d x = \int_{\Omega^*} R(p) d p,
\end{equation} 
holds. Recall that for a function $u$ on $\Omega$, the subdifferential of $u$ at $x \in \Omega$  is defined  by 
\begin{align} \label{normal-mapping}
\partial u (x) = \{ \, p \in \R^d: u(y) \geq u(x) + p \cdot (y-x), \, \text{ for all } \, y \in \Omega\,\}.
\end{align}

We are interested in approximating a convex function $u$ which solves 
\begin{align} \label{m1}
\begin{split}
R(D u(x) ) \det D^2 u(x) &= f(x) \text{ in } \Omega \\
 \partial u(\Omega) &= \Omega^*,
\end{split}
\end{align}
where the first equation is to be interpreted in the sense of Aleksandrov. Specifically,  for a Borel set $E\subset \Omega$, it is required that 
$$
\omega(R,u,E):=  \int_{\partial u(E)} R(p) dp =\int_{E} f(x) dx,
$$ 
where $\omega(R,u,.)$ is the $R$-Monge-Amp\`ere measure associated to $u$. The second equation in \eqref{m1} corresponds to the second boundary condition. When this condition holds, we have
$$
\omega(R,u,\Omega) =  \int_{\partial u(\Omega)} R(p) dp = \int_{\Omega^*} R(p) dp,
$$ and the compatibility condition \eqref{necessary} becomes a necessary condition for the existence of an Aleksandrov solution  $u$ to \eqref{m1}. 

Let $k_{\Omega^*}$ denote the support function of $\Omega^*$, defined by
$$
 k_{\Omega^*}(x) = \sup_{p \in \tir{\Omega^*}} p \cdot x, \text{ for } x \in \R^d.
 $$ 
 Denoting by $a_1^*,\ldots,a_{N^*}^*$ the vertices of $\Omega^*$, we have $k_{\Omega^*}(x) = \max_{a_l^*, l=1,\ldots,N^* } a_l^* \cdot x$. 
 
  It is shown in \cite{Awanou-second-sym} that if we define for $x \notin \Omega$
\begin{equation} \label{extension-00}
u(x) = \inf_{y \in \partial \Omega} u(y) + k_{\Omega^*}(x-y),
\end{equation}
we obtain a convex extension of $u$ to $\R^d$. A convex function $u$ on $\R^d$  that satisfies \eqref{extension-00} is said to have asymptotic cone $K_{\Omega^*}$, where $K_{\Omega^*}$ is the epigraph of $ k_{\Omega^*}$.

Furthermore, \eqref{m1} can be equivalently reformulated as finding a convex function $u$ on $\R^d$ with asymptotic cone $K_{\Omega^*}$ such that
\begin{equation} \label{m02}
R(D u(x) ) \det D^2 u(x) = f(x) \text{ in } \R^d,
\end{equation}
in the sense of Aleksandrov. 
We refer to solutions of \eqref{m02} as Monge-Amp\`ere functions, with anticipated applications in a more general setting \cite{Fu1999}.

Solutions of \eqref{m02} are unique up to an additive constant. For $x^1 \in \Omega$, we may require that $u(x^1)=0$. Given the compatibility condition \eqref{necessary}, \eqref{m02} is then equivalent to finding a convex function $u$ on $\R^d$ with asymptotic cone $K_{\Omega^*}$ such that
in the sense of Aleksandrov
\begin{equation} \label{m2}
R(D u(x) ) \det D^2 u(x) = f(x) + w \, u(x^1) \text{ in } \R^d,
\end{equation}
for some non zero constant $w$. 

For approximating solutions of \eqref{m2}, it has been proposed in \cite{Bakelman1994} to use piecewise linear convex functions. However, if $v$ is a piecewise linear convex function, $\partial v(\Omega)$ is a polygon \cite[Lemma 10]{Awanou-second-sym}. The assumption that $\Omega^*$ is polygonal ensures the possibility $\partial v(\Omega) = \Omega^*$.  
The case where $\Omega^*$ is not necessarily polygonal is discussed in  \cite{Awanou-second-sym}, using polygonal approximations of $\Omega^*$. 

To illustrate the role of $w$ in \eqref{m2} for the purpose of approximating its solution, assume $u \in C^2(\Omega)$ and  let $v \in C^2(\Omega)$. Define
 $$
 J(v) = R(D v(x) ) \det D^2 v(x) - f(x) - w \, v(x^1).
 $$ 
 For the solution $u$ of \eqref{m1} with $u(x^1)=0$, we have
$J(u)= R(D u(x) ) \\
 \det D^2 u(x) - f(x)$, independent of $w$. Moreover if  $u_n \in C^2(\Omega)$ is a sequence of convex functions solving $J(u_n)=0$, the compatibility condition \eqref{necessary} implies that $u_n(x^1)=0$ when $\partial u_n(\Omega) = \Omega^*$. However, if $ \Omega^* \setminus \partial u_n(\Omega)$ has non zero Lebesgue measure, the term $w u_n(x^1)$ will not vanish. If it is known that $u_n$ and its derivatives up to second order  converge uniformly to the corresponding derivatives of $u$, then the term $w u_n(x^1)\to 0$ and thus for $w$ fixed, $u_n(x^1)\to 0$. In summary, the role of $w$ in \eqref{m2} is for the purpose of getting approximations which converge to the solution of \eqref{m1} with $u(x^1)=0$. See also \cite{benamou2017minimal} for a numerical observation. 
 
\subsection{Description of the numerical method} 

Let $h>0$ and define  the orthogonal lattice
$
\mathbb{Z}^d_h = h a + \{\, m h, m \in \mathbb{Z}^d \, \}
$
where $a \in \mathbb{Z}^d$ is an offset that may facilitate the decomposition of the domain used in the discretization of  the  Monge-Amp\`ere equation \eqref{m2d3} below. Put $\Omega_h=\Omega \cap  \mathbb{Z}^d_h$ and let $(r_1,\ldots,r_d)$ denote the canonical basis of $\R^d$. Define the discrete boundary
$$
\Gamma_h = \{ \, x \in \Omega_h  \text{ such that for some } i=1,\ldots,d, x+h r_i \notin \Omega_h \text{ or }  x-h r_i \notin \Omega_h  \,  \}.
$$
Thus $\Gamma_h \subset  \Omega_h$ consists of the lattice points near the boundary of $\Omega$.  For $x \in \Omega_h$, let $V(x)$ be a finite subset of $\mathbb{Z}^d\setminus \{0\}$. We will refer to functions defined on $\mathbb{Z}^d_h$ as mesh functions. 

We define a discrete analogue of the subdifferential as follows: for a mesh function $v_h$ and $x \in \mathbb{Z}^d_h$, 
$$
\partial_V v_h(x) = \{ \, p \in \R^d, p \cdot  (h e) \geq v_h(x) -v_h(x-h e) \, \forall  e \in V(x)  \, \},
$$
and define the discrete  $R$-Monge-Amp\`ere measure: 
$$
\omega_V(R,v_h,E) := \int_{\partial_V v_h(E)} R(p) d p.
$$
We are interested in mesh functions $v_h$ that are $V$-discrete convex, i.e., for all $x \in \Omega_h$ and $e \in V(x)$, 
$$
\Delta_{h e} v_h (x) :=  v_h(x+ h e ) - 2 v_h(x) + v_h(x- h e ) \geq 0. 
$$
The discrete analogue of the asymptotic cone condition \eqref{extension} is defined as follows: for $z \in \mathbb{Z}^d_h \setminus \Omega_h$
\begin{equation} \label{extension}
v_h(z) = \min_{ y \in \Gamma_h } v_h(y) + k_{\Omega^*}(z-y).
\end{equation}
 
Let $\mathcal{C}^{V}$ denote the set of $V$-discrete convex mesh functions, 
and  $\mathcal{C}_0^{V}$ its subset of elements 
 satisfying the extension condition \eqref{extension}. 
 We will denote by $|E|$ the Lebesgue measure of $E$. 
 
We can now describe the discretization of the second boundary value problem we consider in this paper: 
find $u_h \in \mathcal{C}_0^{V}$  such that
\begin{align} \label{m2d3}
\begin{split}
\omega_V(R,u_h,\{\, x \,\})&= \int_{C_x} f(t) dt + w \, u_h(x^1_h), x \in \Omega_h,
\end{split}
\end{align}
where $x^1_h \in \Omega_h, x^1_h \to x^1$ as $h\to 0$ and $(C_x)_{x \in \Omega_h}$ is a partition of $\Omega$ such that:  
\begin{itemize}
\item $C_x \cap \Omega_h = \{ \, x \, \} $
\item $\cup_{x \in \Omega_h} C_x = \Omega$ 
\item $|C_x \cap C_y| \neq 0$ 
\item $|C_x| \neq 0$. 
\end{itemize}
For interior points, one may choose
 $C_x=x+[-h/2,h/2]^d$.  This partition is essential 
for the discretization of the measure with density $f$. 
 
 The unknowns in \eqref{m2d3} are the values $u_h(x), x \in \Omega_h$. For $z \notin \Omega_h$, the value $u_h(z)$ needed for the computation of $\partial_{V} v_h (x) $ is obtained via the discrete extension  \eqref{extension}. 

We prove that for $R \in C(\Omega^*)$ and $w\neq 0$, a subsequence of the damped Newton iterations converges to a solution of \eqref{m2d3}. This yields existence of a solution to \eqref{m2d3}. Moreover, we establish uniqueness of such a solution. Assuming $f>0$ on $\Omega$ and $f\in C(\tir{\Omega})$, we prove the uniform convergence of $u_h$ to the unique Aleksandrov of \eqref{m1} with $u(x^1)=0$. Furthermore, $w u_h(x^1_h) \to 0$ as $h\to0$ and the size of the stencil $V$ approaches its maximum value.

In  \cite{Awanou-second-sym}, a related discretization was considered:  find $u_h \in \mathcal{C}_0^{V_{max}}$  such that
\begin{align} \label{m2d4}
\begin{split}
\omega_V(R,u_h,\{\, x \,\})&= \int_{C_x} f(t) dt, x \in \Omega_h, u_h(x^1_h)=\beta  \in \R,
\end{split}
\end{align}
where the maximal stencil $V_{max}$ is defined in section \ref{preliminaries}. 
When $f>0$ on $\Omega$  and the stencil $V$ is maximal, $u_h$ equals its convex envelope and we obtained in \cite{Awanou-second-sym} existence and uniqueness of the solution of \eqref{m2d4}, as well as convergence of $u_h$ to the unique Aleksandrov of \eqref{m1} with $u(x^1)=\beta$. In the case $V$ is not maximal, under the assumptions that $\Omega$ is a rectangle, $f>0$ on $\Omega$ and $f\in C(\tir{\Omega})$, we proved in \cite{Awanou-second-sym} convergence of a solution $u_h$ of \eqref{m2d4} to the unique Aleksandrov of \eqref{m1} with $u(x^1)=\beta$, as $h\to0$ and the size of the stencil $V$ approaches its maximum value. 

If $u_h$ solves \eqref{m2d3} and $\sum_{x \in \Omega_h} \omega_V(R,u_h,\{\, x \,\}) =\int_{\Omega^*} R(p) dp$, which occurs for instance when $u_h$ equals  its convex envelope and the stencil $V$ is maximal \cite{Awanou-second-sym},  by the compatibility condition \eqref{necessary}, we have 
$$\sum_{x \in \Omega_h} \omega_V(R,u_h,\{\, x \,\}) =\int_{\Omega^*} R(p) dp = \int_{\Omega} f(t) dt = 
\sum_{x \in \Omega_h} \int_{C_x} f(t) dt.
$$
It follows that $u_h(x^1_h)=0$ and our results give the existence and uniqueness of a solution in  $\mathcal{C}_0^{V}$ to \eqref{m2d4} for $\beta=0$. In addition, convergence of the discretization holds for $\Omega$ bounded and convex, $f>0$ on $\Omega$ and $f\in C(\tir{\Omega})$.

\subsection{ Relation of the asymptotic cone approach with other work}

The semi-discrete optimal transport approach for approximating solutions of \eqref{m1}, c.f \cite{kitagawa2016newton}, starts with an approximation of the density $f$ by a sum of $M$ Dirac masses. A convex hull of $M$ points in $\R^{d+1}$ is constructed with a computational complexity O($M \log M$) for $d=2$ and O($M^{\lfloor (d+1)/2 \rfloor}$) for $d \geq 3$, i.e. at least a computational complexity O($M^2$) for $d \geq 3$. Solving the resulting nonlinear system of equations by a damped Newton's method has a worst-case complexity O($M^2$) \cite[Remark 5.5 ]{berman2018convergence}. In summary,  for the  semi-discrete optimal transport approach, both setting up the nonlinear system of equations and solving them has a worst-case complexity  O($M^2$) for $d=2$ and $d=3$. 

To ensure efficiency and guarantee convergence of the iterative method  for solving the discrete equations, the use of power diagrams with a damped Newton's method is advocated in \cite{Yau2013,levy2018notions,kitagawa2016newton,Aurenhammer98,merigot2011multiscale,qiu2020note}. However, this still results in a  worst-case  computational complexity at least  O($M^2$). 

Let $\#V$ denote 
the  maximum of $\{ \, \# V(x), x \in \Omega_h \, \}$ and let $M=\# \Omega_h$ denote the number of mesh points. It is shown in \cite[Section 1.3]{Awanou-second-sym} that setting up the nonlinear equations for the asymptotic cone approach has a computational complexity O($M \#V)$, and their resolution via a  damped Newton's method also has complexity O($M \#V)$. Thus when $\#V$ is chosen as a constant independent of $M$, the asymptotic cone approach has linear complexity. 

The asymptotic cone approach is essentially a finite difference method. The theoretical guarantees of interest include: existence of a solution to the discrete problem,  
uniqueness of a solution to the discrete problem and convergence of the discretization. At least one of these guarantees remains an open problem for the methods proposed in \cite{FroeseSJSC12,Benamou2014,Prins2015,benamou2017minimal,Froese2019,Bonnet-Mirebeau,Brusca-Hamfeldt} where no uniqueness results have been reported. 
 In this paper we prove uniqueness of a discrete solution to the asymptotic cone approach in the general case where it has linear complexity, c.f. section \ref{damped}. We refer to \cite{LindseyRubinstein} for an approach related to discrete optimal transport and to \cite{kawecki2018finite} for a finite element approach. 

\subsection{ Notable features of our analysis of the damped Newton's method}

In  \cite{Mirebeau15,kitagawa2016newton} the goal is to prove the convergence of a damped Newton's method for solving $G(x)=0$ where $G: \mathcal{U} \mapsto \R^M$ and $\mathcal{U}$ is an open subset of $\R^M$. The focus in \cite{kitagawa2016newton} is on convergence rates of the Newton iterates and it is assumed therein that the mapping $G$ is $C^{1,\alpha}, 0 < \alpha <1$. In contrast, both  \cite{Mirebeau15} and the present work  assume only that $G$ is   $C^1$. 

Following  \cite{Mirebeau15}, we also assume that the preimage of any compact set under $G$ is compact, and that the Jacobian matrix $G'(x)$ is invertible for all $x \in \mathcal{U}$.  Invertibility of the Jacobian is similarly assumed in \cite{kitagawa2016newton}.  Due to the different nature of the equations studied in   \cite{Mirebeau15,kitagawa2016newton}, and in this work,  the technical verification of these assumptions necessarily differs. 

A distinctive aspect of this paper is the special treatment of the constraint that the solution lies within the set $ \mathcal{U}$. To handle this, we adopt the approach in \cite{monteiro1999potential}.  

\subsection{Organization of the paper}
In section 
\ref{preliminaries} we give additional preliminaries. The damped Newton's method is introduced in section \ref{damped-general} in a general setting. In section \ref{damped} we give its convergence analysis for \eqref{m2d3} and discuss the existence and uniqueness of solutions to \eqref{m2d3}.  In section \ref{limit}, we build on convergence results obtained in \cite{Awanou-second-sym} to examine the convergence of $w u_h(x^1_h)$.  We conclude the paper with some numerical experiments. 


\section{Preliminaries } \label{preliminaries}

We will use the notation $C$ for a generic constant and $|| \cdot ||$ for the Euclidean norm. 
We first describe the extended mesh needed for the computation of $\partial_{V} v_h (x) $ in \eqref{m2d3}. A stencil $V$ is a set valued mapping from $\Omega_h$ to the set of finite subsets of $\mathbb{Z}^d\setminus \{0\}$. Recall that a subset $W$ of $\mathbb{Z}^d$ is symmetric with respect to the origin if $\forall y \in W, -y \in W$. We define $V_{min}$ to be a finite subset of $\mathbb{Z}^d\setminus \{0\}$ which is symmetric with respect to the origin, contains the elements of the canonical basis of $\R^d$, and contains a vector parallel to a normal vector to each facet of the polygonal domain $\Omega^*$. Define the extended mesh and maximal stencil as follows: 
\begin{align*}
\Omega_{ext} & = \Omega_h \cup \{ \, x+ h e: x \in \Omega_h, e \in V_{min} \, \} \\
V_{max}(x) & = \{ \, e \in \mathbb{Z}^d \setminus \{0\}, \exists y \in \Omega_{ext}, y = x+ h e \, \}.
\end{align*}
We assume:
$$
V_{min} \subset V(x) \subset V_{max}(x), \text{ for all } x \in \Omega_h,
$$
and that $V(x) \subset \mathbb{Z}^d \setminus \{ \, 0 \, \}$ is symmetric with respect to the origin for $x \in \Omega_h$. Moreover, we assume that if $a, b \in V(x)$ and $a=r b$ for some scalar $r$, then $r=-1$.

\begin{lemma}  \cite[Theorem 4]{Awanou-second-sym} \label{exi-V-max}
If $f>0$ on $\Omega$, 
then solutions of \eqref{m2d4} for $V=V_{max}$ are unique up
to an additive constant. 
\end{lemma}

\begin{lemma} \cite[Lemma 2]{Awanou-second-sym} \label{pre-lip-lem}
There exists a constant $C$ independent of $v_h \in \mathcal{C}_0^{V}$ and $h$, such that for all $x, y \in \Omega_h$
$$ 
| v_h(x)-v_h(y) | \leq C ||x-y||.
$$
\end{lemma}

\begin{lemma} \label{inc-sub-lem}
Let $v_h \in \mathcal{C}_0^{V}$. 
Then for $x \in \Omega_h$, $\partial_V v_h (x) \subset  \Omega^*$. 
\end{lemma}
This  follows from \cite[Lemma 2]{Awanou-second-sym} using the assumption $V_{\min} \subset V(x)$. The following lemma was proved in Part 1 of  \cite[Theorem 5]{Awanou-second-sym}
\begin{lemma} \label{key-for-u}
Let $u_h, v_h \in  \mathcal{C}_0^{V}$ and assume that there exists $z \in \Omega_h$ such that $u_h(x) - v_h(x) \geq u_h(z) - v_h(z)$ for all $x \in \Omega_h$. Then: $ \omega_V(R, u_h,\{\, z\, \})  \geq  \omega_V(R, v_h,\{\, z\, \}) $. 
\end{lemma}

\begin{lemma} \label{support}
Let $Y \subset \R^d$ be a closed convex set such that $\tir{\Omega^*} \subset \mathring{Y}$. Then:
$$
k_Y(x) > k_{ \Omega^* }(x), \quad  \forall x \in R^d.
$$
\end{lemma}

\begin{proof}
From the definition of support functions, we have $k_Y(x) \geq  k_{ \Omega^* }(x)$ for all $x \in R^d$. Fix $x \in R^d$ and let
$z \in \partial \Omega^* \cap \{ \, y \in \R^d:  y \cdot x =  k_{ \Omega^* }(x) \, \}$. 
Since $\tir{\Omega^*} \subset \mathring{Y}$, there exists $\epsilon>0$ such that the ball $B(z,\epsilon)$ of center $z$ and radius $\epsilon$ is contained in $Y$. As $z \in \partial \Omega^*$, there exists $y \in B(z,\epsilon) \setminus \tir{\Omega^*}$. We have $y \in Y$ and
$y \cdot x >  k_{ \Omega^* }(x) $. It follows that $k_Y(x) \geq y \cdot x >  k_{ \Omega^* }(x)$. 
\end{proof}

\section{The  damped Newton's method } \label{damped-general}

Let $\mathbb{U}$ be an open subset of $\R^N$ and let $G: \tir{ \mathbb{U} } \to \R^N$. We assume that 
 $G \in C^1(\mathbb{U}_{ },\R^N)$ and $G$ is a proper map with $\det G'(v)  \neq 0$ for all $v \in \mathbb{U}_{  }$. 
 
Modifications of the generic damped Newton's method needed to ensure convergence to a zero of
 $G$ within
$\tir{\mathbb{U} }$ were developed in \cite{monteiro1999potential}. We adopt their framework and  make the following assumption

\begin{ass} \label{a1}
There exists a closed convex set $T \subset \R^N$ such that
\begin{enumerate}
\item[(a)] $0 \in T$
\item[(b)] the (open) set $\mathbb{U}_{ \mathcal{I}} := G^{-1}( \mathring{T} ) \cap  \mathbb{U}$ is nonempty
\item[(c)] the set $G^{-1}(  \mathring{T}  ) \cap \partial \mathbb{U}$ is empty. 
\end{enumerate}
\end{ass}


Let $v^0 \in \mathbb{U}_{ \mathcal{I} }$. Given the current iterate $v^k$, the next iterate is sought along the path
$$
p^k(\tau) = v^k - \tau \big(G'(v^k)  \big)^{-1} G(v^k).
$$
Fix parameters $\delta \in (0,1)$ and choose $\rho \in (0,1)$, e.g. $\rho=1/2$. 
\begin{algorithm}
  \caption{ A damped Newton's method}\label{euclid}
  \begin{algorithmic}[1]
    \State Choose $v^0 \in \mathbb{U}_{ \mathcal{I} }$ and set $k=0$
     
  \State \textbf{If}  $G(v^k)=0$ 
  \textbf{stop} \label{marker}  
    
    \State Let $i_k$ be the smallest non-negative integer $i$ such that: 
    $$
   p^k(\rho^i) \in \mathbb{U}_{ \mathcal{I} } \text{ and } || G(p^k(\rho^i)) ||^2 \leq (1 - \delta \rho^i) || G(v^k) ||^2.
    $$
    Set $v^{k+1} = p^k(\rho^{i_k})$
    \State $k\gets k+1$ and \Goto{marker}.
 \end{algorithmic}
\end{algorithm}

Let $q: \R^N \to \R$ defined by $q(x) = ||x||^2$ and define the associated function $\chi: \mathbb{U} \to \R$ by
$$
\chi(v) = q(G(v)) = ||G(v)||^2.
$$
Let $d^k = -  \big(G'(v^k)  \big)^{-1} G(v^k)$ denote the Newton  direction. We have 
\begin{multline*}
\chi'(v^k) (d^k) = q'(G(v^k)) (G'(v^k) d^k)= q'(G(v^k)) (-G(v^k)) = -2 G(v^k) \cdot G(v^k) \\= -2 || G(v^k)||^2 <0.
\end{multline*}

Thus, $d^k$ is a descent direction for $\chi(v)$ and since $ \mathbb{U}_{ \mathcal{I} }$ is open, the line search in the algorithm is well defined. 

The following result is a consequence of \cite[Theorem 3]{monteiro1999potential}. Therein, the function $q$ is referred to as a potential and the damped Newton's method thus induces a potential reduction. 

\begin{thm} \label{damped-th}
Assume that Assumption \ref{a1} holds. Let $G \in C^1(\mathbb{U}_{ },\R^N)$ be a proper map with 
$\det G'(v)  \neq 0$ for all $v \in \mathbb{U}_{  }$. For any initial guess $v^0$ in $\mathbb{U}_{\mathcal{I}}$, there exists a subsequence $v^{k_l}$ of the damped Newton's method iterates $v^k$ that converges to a zero $v$ of $G$ in $\tir{\mathbb{U}}$. 
\end{thm}

We note that by Assumption \ref{a1}, $v^k \in 
   \mathbb{U}_{ \mathcal{I} }$ for all $k$. Thus $G(v^k) \in  \mathring{T} $ for all $k$. If for example $0 \notin  \mathring{T} $, 
we can only guarantee convergence of the sequence in $\tir{\mathbb{U}}$ as $\partial  \mathbb{U}_{ \mathcal{I} } \subset \tir{\mathbb{U}}$. We also note that under the assumptions of Theorem \ref{damped-th}, if $G$ has a unique zero $v$ in $\tir{\mathbb{U}}$, the whole damped Newton's sequence converges to $v$.

\section{Convergence of the  damped Newton's method for the discretization} \label{damped}

In this section we prove existence and uniqueness of a solution to \eqref{m2d3}. 
Recall the extension formula \eqref{extension}. For $x \in \Omega_h$ and $e \in V(x)$ such that $x+ h e \notin \Omega_h$, we define
$$
\Gamma(x+he)= \argmin_{ y \in \Gamma_h }  v_h(y) + k_{\Omega^*}(x-y),
$$
where $ \Gamma_h$ is the discrete boundary and $k_{\Omega^*}$ the support function of $\Omega^*$.
A priori, $\Gamma(x+he)$ is multi-valued. We assume that for the implementation a unique choice is made for pairs $(x,e)$ such that $x+ h e \notin \Omega_h$. If $x+ h e \in \Omega_h$, we put $\Gamma(x+he)=x+h e$.

We define
$$
\mathcal{M}_h = \{ \, x + h e, x \in \Omega_h \text{ and } e \in V(x) \cup\{\,0\,\}\, \}.
$$

Let $M$ and $N$ denote the cardinality of $\Omega_h$ and $\mathcal{M}_h$ respectively.
We denote the points of $\Omega_h$ by $x^i, i=1,\ldots,M$ and the points of $\mathcal{M}_h\setminus \Omega_h$ by
 $x^i, i=M+1,\ldots,N$. In this section $h$ is fixed. Hence, for convenience, we do not indicate the dependence of $x^1$ on $h$.

The set $\mathcal{F}_h$ of mesh functions on $\mathcal{M}_h$, is identified with $\R^{N}$ by mapping $v_h$ to
$ (v_h(x_i))_{i=1,\ldots,N}$. 
We consider a map $G: \mathcal{F}_h \to \R^{N}$ defined by
\begin{align} \label{Gg}
\begin{split}
G_i(v_h) & =  \omega_V(R,v_h,\{\, x^i \,\})- \int_{C_{x^i}} f(t) dt - w \, v_h(x^1) \text{ for } i=1,\ldots,M \\
G_i(v_h) & = -v_h(x^i) + v_h(y) + k_{\Omega^*}(x-y), y= \Gamma(x^i), i=M+1,\ldots,N,
\end{split}
\end{align}
and we recall that $w \neq 0$ is a constant. 
It will be convenient to consider the mapping $G$ for $w=0$ as well.

Next, we choose the closed convex set $T$ from Assumption \ref{a1} and the open subset $\mathbb{U}$ of $\R^N$. Put
\begin{equation} \label{t}
T:=  \{ \, v=(v_i)_{i=1,\ldots,N}, v_i \geq 0,  i=1,\ldots,M, v_i \leq 0, i=M+1,\ldots,N \, \}.
\end{equation} 
For $v \in \R^J$, we will use the notation $v\geq0$ if $v_i\geq 0$ for all $i$. Similarly, we will also use $v>0, v\leq 0$ or $v<0$. Define
\begin{multline*}
\mathbb{U} = \{ \, v_h \in \mathcal{F}_h, \omega_V(R,v_h,\{\, x \,\}) >0, 
w v_h(x^1) +  \int_{C_{x}} f(t) dt - \omega_V(R,v_h,\{\, x \,\}) >0, \\ \, \forall x \in \Omega_h  \text{ and }
v_h(x) - v_h( \Gamma(x)) - k_{\Omega^*}(x- \Gamma(x)) >0, \, \forall x \in \Gamma_h
\, \}.
\end{multline*}
We will need below the set
$$
\mathbb{U}_0 = \{ \, v_h \in \mathcal{F}_h, \omega_V(R,v_h,\{\, x \,\}) >0, \, \forall x \in \Omega_h
\, \}.
$$

\begin{rem} \label{imply-disc-convex}
The condition $\omega_V(R,v_h,\{\, x \,\}) >0, \, \forall x \in \Omega_h$ implies that  $v_h$ is strictly $V$-discrete convex in the sense that $\Delta_{h e} v_h(x) >0$ for all $e \in V(x)$. This follows from $\partial_V v_h(x) \subset \{ \, p \in \R^d,  v_h(x) - v_h(x-he) \leq p \cdot (h e) \leq v_h(x+he) - v_h(x)
\, \}$. If $\Delta_{h e} v_h(x) =0$, $\partial_V v_h(x) $ is contained in a set of measure 0, which gives $\omega_V(R,v_h,\{\, x \,\}) =0$. In particular, elements of 
$\mathbb{U}$ are strictly $V$-discrete convex. 
\end{rem}



The first goal of this section is to verify the assumptions of Theorem \ref{damped-th} for the equation $G(v_h)=0$ with $G$ given by \eqref{Gg}. 
The second goal of this section is the proof of uniqueness of solutions to \eqref{m2d3} which are in $\mathbb{U}_0$.
We will first prove the following theorem

\begin{thm} \label{main0} Assume that $ R \in C(\Omega^*)$ and $w \neq 0$. Given an initial guess $v_h^0 \in \mathbb{U}_{\mathcal{I}}$,  there is a subsequence $v_h^{k_l}$ of the damped Newton's method which converges to a solution $u_h$ of \eqref{m2d3}. 
\end{thm}

The proof of Theorem \ref{main0} proceeds in several steps. We first show that the set $\mathbb{U}$ is open and that Assumption \ref{a1} holds. The $C^1$ continuity of $G$ is established in Theorem \ref{C1-continuous} and the invertibility assumption in Theorem \ref{det} for $w \neq 0$. 
Finally, in Theorem \ref{proper}, we prove that the mapping $G$ is proper for $w \neq 0$. We are then in a position to give the proof of Theorem \ref{main0}.

\subsection{ The set $\mathbb{U}$ is open and Assumption \ref{a1} holds} \label{non-empty}

\begin{thm}
The set $\mathbb{U}$ is open.
\end{thm}
\begin{proof}
For $v_h$ and $w_h$ in $\mathcal{F}_h$, we define
$$
|w_h-v_h|_{\infty} := \max \{ \, |w_h(x) - v_h(x)|, x \in \mathcal{M}_h \, \}.  
$$
Let $v_h \in \mathbb{U}$ and $x \in \Omega_h$. We have $\omega_V(R,v_h,\{\, x \,\}) >0$ and $w v_h(x^1) +  \int_{C_{x}} f(t) dt - \omega_V(R,v_h,\{\, x \,\}) >0$. We show that there exists $\epsilon>0$ such that if $|v_h-w_h|_{\infty} < \epsilon$, then 
$\omega_V(R, w_h,\{\, x \,\}) >0$ and $w w_h(x^1) +  \int_{C_{x}} f(t) dt - \omega_V(R,w_h,\{\, x \,\}) >0$. 

 Let $\epsilon>0$ such that $|v_h-w_h|_{\infty} < \epsilon$. If $p \in \partial_V w_h(x)$, for all $e \in V(x)$, we have
$$
p \cdot (he) \leq w_h(x+he) - w_h(x) \leq v_h(x+he) + \epsilon - v_h(x) + \epsilon.
$$
Choosing $e=\pm r_i$ where $\{ \, r_1,\ldots,r_d \, \}$ is the canonical basis of $\R^d$, we conclude that 
$ \partial_V w_h(x)$ is bounded. We have
\begin{multline*}
\omega_V(R,w_h,\{\, x \,\})  - \omega_V(R,v_h,\{\, x \,\}) = \int_{\partial_V w_h(x)} R(p) dp - \int_{\partial_V v_h(x)} R(p) dp  \\ = \int_{\partial_V w_h(x) \setminus \partial_V v_h(x) } R(p) dp -  \int_{\partial_V v_h(x) \setminus \partial_V w_h(x) } R(p) dp.
\end{multline*}
If $p \in \partial_V w_h(x) \setminus \partial_V v_h(x)$ there exists $e \in V$ such that 
$$
v_h(x+he)  - v_h(x)  < p \cdot (h e) \leq w_h(x+he) - w_h(x). 
$$
Since
$$
| (w_h(x+he) - w_h(x)) - (v_h(x+he)  - v_h(x)) | \leq 2 \epsilon,
$$
and $\partial_V w_h(x)$ is bounded, we conclude that there is a constant $C$ which depends on $v_h$ and $h$ such that
$| \partial_V w_h(x) \setminus \partial_V v_h(x) | \leq C \epsilon$. Similarly, 
$| \partial_V v_h(x) \setminus \partial_V w_h(x) | \leq C \epsilon$ for a constant $C$ which depends on $v_h$ and $h$. 

We have
$\lim_{\epsilon \to 0}  \int_{\partial_V w_h(x) \setminus \partial_V v_h(x) } R(p) dp
= \lim_{\epsilon \to 0} \int_{\partial_V v_h(x) \setminus \partial_V w_h(x) } R(p) dp =0$ since $R$ is integrable. We conclude that 
$$
\lim_{\epsilon \to 0} | \omega_V(R,w_h,\{\, x \,\})  - \omega_V(R,v_h,\{\, x \,\}) | =0 \text{ and }
$$
$$
\lim_{\epsilon \to 0} |(w w_h(x^1) -  \omega_V(R,w_h,\{\, x \,\}) ) - ( w v_h(x^1) -  \omega_V(R,v_h,\{\, x \,\})  ) | = 0.
$$
Therefore, there exists $\epsilon>0$ such that
\begin{align} \label{01}
| \omega_V(R,w_h,\{\, x \,\})  - \omega_V(R,v_h,\{\, x \,\}) | & < \omega_V(R,v_h,\{\, x \,\}) /2,
\end{align}
and 
\begin{multline*}
| (w w_h(x^1)- \omega_V(R,w_h,\{\, x \,\}))  - (w v_h(x^1)- \omega_V(R,v_h,\{\, x \,\})) |  < \frac{1}{2} \bigg(w v_h(x^1) \\
 +  \int_{C_{x}} f(t) dt - \omega_V(R,v_h,\{\, x \,\}) \bigg).
\end{multline*}
It follows that $ \omega_V(R,w_h,\{\, x \,\}) >   \omega_V(R,v_h,\{\, x \,\})/2 >0 $ and 
\begin{multline}  \label{02}
w w_h(x^1)- \omega_V(R,w_h,\{\, x \,\}) +\int_{C_{x}} f(t) dt) > \frac{1}{2} \left(w v_h(x^1)- \omega_V(R,v_h,\{\, x \,\}) \right) \\ + \frac{1}{2}  \int_{C_{x}} f(t) dt>0.
\end{multline}

Since $\Omega_h$ is finite, we can choose $\epsilon_1>0$ 
such that \eqref{01} and \eqref{02} hold for all $x \in \Omega_h$. 

Furthermore, let $\epsilon_2= 1/2 \min \{ \, v_h(x) - v_h( \Gamma(x)) - k_{\Omega^*}(x- \Gamma(x)) >0:  x \in  \mathcal{M}_h \setminus \Omega_h \, \}$. For $\epsilon=
\min\{\, \epsilon_1, \epsilon_2 \, \}$, we have in addition to \eqref{01} and \eqref{02}, for all $x \in  \mathcal{M}_h \setminus \Omega_h$
\begin{multline*}
w_h(x) - w_h( \Gamma(x)) - k_{\Omega^*}(x- \Gamma(x)) > v_h(x) - v_h( \Gamma(x)) - k_{\Omega^*}(x- \Gamma(x)) - 2 \epsilon >0.
\end{multline*}
This concludes the proof. 
\end{proof}

Next, we verify  Assumption \ref{a1}. We have by construction $0 \in  T$ with $T$ defined by \eqref{t}. We first show that $G^{-1}(  \mathring{T}  ) \cap \partial \mathbb{U}$ is empty.

If $v_h \in  \partial \mathbb{U}$, one, two or all of the following hold:
\begin{itemize}

\item $ \omega_V(R,v_h,\{\, x \,\}) =0$ for some $x \in \Omega_h$
\item $w \, v_h(x^1) +  \int_{C_{x}} f(t) dt- \omega_V(R,v_h,\{\, x \,\})  =0$ for some $x \in \Omega_h$
\item $v_h(x) - v_h( \Gamma(x)) - k_{\Omega^*}(x- \Gamma(x))=0 $ for some $x \in \mathcal{M}_h \setminus \Omega_h$. 
\end{itemize}
If the second or third condition hold, $G_i(v_h) =0$ for some $i=1,\ldots,N$ and thus $G(v_h) \notin \mathring{T} $. Assume that $ \omega_V(R,v_h,\{\, x^i \,\}) =0$ for some $x^i \in \Omega_h$, $i=1,\ldots,M$ and we don't already have $w v_h(x^1) +  \int_{C_{x^i}} f(t) dt- \omega_V(R,v_h,\{\, x^i \,\})  =0$. 
Since $v_h \in \partial \mathbb{U}$, we conclude that $0 < w v_h(x^1) +  \int_{C_{x^i}} f(t) dt- \omega_V(R,v_h,\{\, x^i \,\}) = w v_h(x^1) +  \int_{C_{x^i}} f(t) dt$. It follows that $G_i(v_h) = - \int_{C_{x^i}} f(t) dt - w \, v_h(x^1)<0$. Therefore  $G(v_h) \notin \mathring{T} $. 
This proves that $G^{-1}(  \mathring{T}  ) \cap \partial \mathbb{U} = \emptyset$. 

Finally, we show that $G^{-1}( \mathring{T} ) \cap  \mathbb{U}$ is nonempty. Let $Y$ be a closed convex polygon such that $\tir{\Omega^*}$ is contained in the interior of $Y$ and $\int_{Y \setminus \Omega^*} \tilde{R}(p) dp = |\Omega|$, where 
$ \tilde{R}>0$ is an extension of $R$ to $Y$. We have the compatibility condition $\int_{Y}  \tilde{R}(p) dp = \int_{\Omega}  \big( f(t) + 1 \big)dt$.

By Lemma \ref{exi-V-max}, there exists $w_h \in \mathcal{C}^{V_{max}}_0$
such that $w_h(x^1_h)=0$ and 
\begin{align*}
 \omega_{V_{max}}(  \tilde{R} ,w_h,\{\, x \,\}))  & =  \int_{C_{x}} \big( f(t) + 1 \big)dt, \forall x \in \Omega_h \\
 w_h(x) & = w_h( \Gamma(x)) + k_{Y}(x- \Gamma(x)), \, \forall x \in \mathcal{M}_h \setminus \Omega_h. 
\end{align*}
We have $ \mathcal{C}^{V_{max}}_0 \subset  \mathcal{C}^{V }$ and $ \omega_{V }(R,w_h,\{\, x \,\}) 
=  \omega_{V }(\tilde{R},w_h,\{\, x \,\}) \geq  \omega_{V_{max}}( \tilde{R} ,w_h,\{\, x \,\}) \geq |C_x| >0$ for all 
$x \in \Omega_h$, where we used $\partial_V w_h(\Omega_h) \subset \Omega^*$ (Lemma \ref{inc-sub-lem}) and $R=\tilde{R}$ on $\Omega^*$. 
We have for all 
$x \in \Omega_h$
\begin{align*}
 \int_{C_{x}} f(t)  dt +|C_x| -  \omega_{V }(R,w_h,\{\, x \,\}) \leq 0. 
\end{align*}
Let $\mu=\min \{ \, \int_{C_{x}} f(t)  dt +|C_x| -  \omega_{V }(R,w_h,\{\, x \,\}): x \in \Omega_h
\, \}$ and define $z_h$ by $z_h(x)=w_h(x)+( 2 |C_x|- \mu)/w$. Using $w \, w_h(x^1_h)=0$, we have for all $x \in \Omega_h$
$$
w z_h(x^1_h)= -\mu+2 |C_x| \geq  -\int_{C_{x}} f(t)  dt - |C_x|+  \omega_{V }(R,z_h,\{\, x \,\}) +2 |C_x|,
$$
which gives
$$
w z_h(x^1_h) + \int_{C_{x}} f(t)  dt -  \omega_{V }(R,z_h,\{\, x \,\}) \geq  |C_x| >0.
$$
We conclude that $G_i(z_h) >0, i=1,\ldots,M$. Moreover, for $ x \in \mathcal{M}_h \setminus \Omega_h$
\begin{multline*}
-w_h(x) + w_h( \Gamma(x)) + k_{\Omega^*}(x- \Gamma(x)) = -w_h(x) + w_h( \Gamma(x)) \\
+ k_{Y}(x- \Gamma(x)) + k_{\Omega^*}(x- \Gamma(x)) - k_{Y}(x- \Gamma(x)) = 
 k_{\Omega^*}(x- \Gamma(x)) - k_{Y}(x- \Gamma(x)) <0,
\end{multline*}
where we used Lemma \ref{support}. We conclude that $G_i(z_h) <0, i=M+1,\ldots,N$. We have proved that $G^{-1}( \mathring{T} ) \cap  \mathbb{U}\neq \emptyset$.

\subsection{$C^1$ continuity of the mapping $G$}

Let $\# W$ denote the cardinality of the set $W$. Given $\lambda \in \R^{\# W }$, we write $\lambda = (\lambda_a)_{a \in W}$ by an abuse of notation, instead of the more familiar notation
$\lambda=(\lambda_j)_{j=1,\ldots,\# W}$. 

We fix $i$ in $\{ \, 1,\ldots, M \, \}$ and for $\lambda \in  \R^{\# V(x^i)}$, let 
$$
Q(\lambda) = \{ \, p \in \R^d, p \cdot e \leq \lambda_e, \ \forall e \in V(x^i) \, \},
$$
where for simplicity we do not mention the dependence of $Q(\lambda)$ on the index $i$.

Consider the mapping $S$ defined by $S(\lambda): = \int_{Q(\lambda)} R(p) d p$. 
For a given index $j$ we are interested in the variations of $G_i(v)$ with respect to $v(x^j)$, i.e. the derivative at $v(x^j)$ of the application which is the composite of $S$ and the mapping
$$
r \mapsto \lambda=\bigg( \frac{w(x^i+h a) -w(x^i)}{h}  \bigg)_{a \in V(x^i)},
$$
where  $w$ is the mesh function defined by
\begin{align*}
w(x) = v(x), x \neq x^j, \quad w(x^j) = r.
\end{align*}

Let  $z_i \in \R^d, i=1,\ldots,N$. Let $(r_1,\ldots,r_N)$ denote the canonical basis of $\R^N$. Given $\lambda \in \R^N$, define
$$
\tilde{Q}(\lambda) := \{ \, x \in \R^d,  x \cdot z_i \leq \lambda_i, i =1,\ldots,N \, \}.
$$
Assume that the vectors $z_i$ are chosen such that the polytope $\hat{Q}(\lambda)$ is bounded. We will need the following  lemma  \cite[Lemma 16]{Meyron2018}. 

\begin{lemma} \label{mey2}
Let $\rho: \R^d \to \R$ be a continuous function. Define
$$
\tilde{S}(\lambda): = \int_{\tilde{Q}(\lambda)} \rho(p) d p
$$

\begin{enumerate}

\item If $z_i \neq 0$ for all $i$, $\tilde{S}$ is continuous on $\R^d$. 

\item $\forall R \geq 0$, $\exists Q_R \subset \R^d$ compact such that $\forall \lambda' \in \R^N$, $\max_i | \lambda'  - \lambda| \leq R$ implies $\tilde{Q}(\lambda')  \subseteq Q_R$. 

\item There exists a function $\eta_R: \R^+ \to \R^+$ satisfying $\lim_{s \to 0} \eta_R(s) =0$ such that for all $x, y \in Q_R$, $|\rho(x) - \rho(y)| \leq \eta_R(||x-y||)$. 

\item Let $i_0 \in \{ \, 1,\ldots, N\, \}, \lambda \in R^N$ and $t\geq 0$. We have
$$
\frac{1}{t} (\tilde{S}(\lambda+ t \, r_{i_0}) - \tilde{S}(\lambda)) = \frac{1}{t} \int_0^t g_{i_0} (\lambda + s r_{i_0} ) ds,
$$
where $g_{i_0} (\tir{\lambda} ): = 1/||z_{i_0}|| \int_{\tilde{Q}(\tir{\lambda}) \cap \{ \, x \in \R^d, x \cdot z_{i_0} = \tir{\lambda}_{i_0} \, \}} \rho(p) dp.$

\item Let $\Pi_{i_0}$ denote the orthogonal projection onto the hyperplane orthogonal to $z_{i_0}$. Put
 $P_{i_0}(\lambda) = \Pi_{i_0} \big( \tilde{Q}(\lambda) \cap \{ \, x \in \R^d, x \cdot z_{i_0} = \lambda_{i_0} \, \}  \big)$. We have 
 $$
 P_{i_0}(\lambda) = \{ \, y \in \R^d, y \cdot z_{i_0} = 0 \text{ and } y \cdot \Pi_{i_0}(z_i) \leq \lambda_i -  \lambda_{i_0} z_i \cdot z_{i_0}/||z_{i_0}||^2, i \neq i_0   \, \}. 
 $$
 
 \item Define $\rho_{\lambda} (y) = \rho(y + \lambda_{i_0} z_{i_0}/||z_{i_0}||^2 )$. We have 
 $||z_{i_0}||  \, | g_{i_0} (\lambda ) -  g_{i_0} (\lambda' ) | \leq |A_1| + |A_2|$, where
 \begin{align*}
 A_1 & : = \int_{ P_{i_0}(\lambda) } (\rho_{\lambda} (y)  - \rho_{\lambda'} (y) ) dy \\
  A_2 & : = \int_{ P_{i_0}(\lambda) }  \rho_{\lambda'} (y) ) dy - \int_{ P_{i_0}(\lambda') }  \rho_{\lambda'} (y) ) dy. 
 \end{align*}
 
 \item As $\lambda' \to \lambda$, $|A_1| \to 0$ by (3) above and if the vectors $\Pi_{i_0}(z_i)$ are non zero, $|A_2| \to 0$ using (1) above, giving the $C^1$ continuity of $\tilde{S}$. 

\end{enumerate}
\end{lemma}

\begin{lemma} \label{Meyron}
Let $W$ be a set of non zero vectors which contains the canonical basis of $\mathbb{R}^d$ with the property that for $e$ and $f$ in $V$, $e=r f$ for a scalar $r$ if and only if $r=-1$. Assume furthermore that $W$ is symmetric with respect to the origin. Let $R$  be continuous on $\Omega^*$ and assume that for $\lambda \in \R^{\# W}, Q(\lambda)  \subset \Omega^*$. 
Then the mapping
$$
\lambda \mapsto  S(\lambda): = \int_{Q(\lambda)} R(p) d p,
$$
is $C^1$ continuous on $\{ \, \lambda \in \R^{\# W}, \lambda_{-a} + \lambda_a > 0, \forall a \in W \, \}$ with
$$
\frac{\partial S }{\partial  \lambda_a} = \frac{1}{||a||} \int_{Q(\lambda) \cap \{ \, p \in \R^d, p \cdot a = \lambda_a\, \} } R(p) d p.
$$
\end{lemma}

\begin{proof}
Since $W$ contains the canonical basis of $\mathbb{R}^d$, the polytope $Q(\lambda)$ is bounded. Let $e, a \in W$ such that $e \neq a$. Let $\Pi_a$ denote the orthogonal projection onto the hyperplane orthogonal to $a$. By Lemma \ref{mey2} we have
\begin{multline*}
 \Pi_a \big( Q(\lambda) \cap \{ \, x \in \R^d, x \cdot a = \lambda_{a} \, \} 
 \big) = \\
 \{ \, y \in \R^d, y \cdot a = 0 \text{ and } y \cdot \Pi_{a}(e) \leq \lambda_e -  \lambda_a \, e \cdot a/ ||a||^2, e \neq a   \, \}. 
 \end{multline*}
 
 If $e=-a$, $ \lambda_{-a} -  \lambda_a \, (-a) \cdot a/ ||a||^2= \lambda_{-a} + \lambda_a > 0$ by assumption. We then have trivially $y \cdot \Pi_{a}(-a) = 0 <  \lambda_{-a} + \lambda_a$. Thus
 \begin{multline*}
 \Pi_a \big( Q(\lambda) \cap \{ \, x \in \R^d, x \cdot a = \lambda_{a} \, \} 
 \big) = \\
 \{ \, y \in \R^d, y \cdot a = 0 \text{ and } y \cdot \Pi_{a}(e) \leq \lambda_e -  \lambda_a \, e \cdot a/ ||a||^2, e \notin \{ \, a, -a \, \}   \, \}. 
 \end{multline*}
 
 If $e \notin \{ \, a, -a \, \} $, $e$ and $a$ are independent which implies $\Pi_{a}(e) \neq 0$. The $C^1$ continuity of $S$ then follows from Lemma \ref{mey2}
\end{proof}

Recall that for $x \in \Omega_h$, $V(x)$ satisfies the assumptions on $W$ in Lemma \ref{Meyron}. Define for $v \in \mathbb{U}_{}$
$$
\lambda_{ha}(v)(x) = \frac{v(x+ h a) - v(x)}{h}.  
$$
Recall that $\partial_V v(x) = \{ \, p \in \R^d, p \cdot (h e) \leq \lambda_{h e}(v)(x), \forall e \in V(x)  \, \}$. Thus, if we put 
$\lambda \equiv (\lambda_{h e}(v)(x))_{e \in V(x)}$ we have
\begin{equation} \label{qlsubd}
\partial_V v(x) = Q( \lambda).
\end{equation}
We omit the dependence of $\lambda$ on $x$ as it will be clear from the context. 
\begin{thm} \label{C1-continuous}
The mapping $G$ is $C^1$ continuous on $\mathbb{U}_{0}$ for $R$ continuous on $\Omega^*$.  
\end{thm}

\begin{proof} 
By Lemma \ref{inc-sub-lem}, for $v \in \mathbb{U}_{0}$, $\partial_V v(x) \subset \Omega^*$ for all $x \in \Omega_h$. Since elements of $\mathbb{U}_{0}$ are strictly $V$-discrete convex, we have for $v \in \mathbb{U}_{0}$, $\lambda_{ha}(v)(x) + \lambda_{-ha}(v)(x) >0$ for all $x \in \Omega_h$ and $a \in V(x)$. We note that the mapping $v_h \mapsto v_h(x^1)$ is  $C^1$ continuous. On the other hand, the mapping
$v_h \mapsto  \omega_V(R,v_h,\{\, x^i \,\})$ is the composite of $S$ and the mapping $v_h \mapsto  \bigg( \frac{v_h(x^i+h a) -v_h(x^i)}{h}  \bigg)_{a \in V(x^i)}$. By Lemma \ref{Meyron} the functional $G_i$ is $C^1$ continuous on $\mathbb{U}_{0}$, 
$i=1,\ldots,M$. The $C^1$ continuity of $G_i$ for $i=M+1,\ldots,N$ is immediate. This concludes the proof. 
\end{proof}

Given $i \in \{ \, 1, \ldots, M \, \}$, $\partial_V v(x^i)$ and hence $\omega_V(R,v,\{\, x^i \,\})$ depends on $x^i, x^i+h e, e \in V(x^i)$. 
We have for $i, j \in \{ \, 1, \ldots, N \, \}$
\begin{multline} \label{j1}
\text{if } x^j \notin \{ x^i + h e, e \in V(x^i) \}  \text{ and } j \neq i, \\
 \partial \omega_V(R,v,\{\, x^i \,\}) /\partial  v(x^j)  =0.
\end{multline}

For $e \in V(x^i)$ 
\begin{align} \label{j4}
\frac{\partial \omega_V(R,v,\{\, x^i \,\}) }{\partial v(x^i + h e)} =  \frac{1}{h ||e||} \int_{Q(\lambda) \cap \{ \, p \in \R^d, p \cdot (h e) = \lambda_{h e}\, \} } R(p) d p. 
\end{align}

Next, 
\begin{align} \label{j5}
\frac{\partial \omega_V(R,v,\{\, x^i \,\})   }{\partial v(x^i)} = - \sum_{e \in V(x^i)} \frac{1}{h ||e||} \int_{Q(\lambda) \cap \{ \, p \in \R^d, p \cdot (h e) = \lambda_{h e}\, \} } R(p) d p.
\end{align}

Put
$$
L_e = \frac{1}{h ||e||} \int_{Q(\lambda) \cap \{ \, p \in \R^d, p \cdot (h e) = \lambda_{ h e}\, \} } R(p) d p.
$$
For $x \in \Omega_h$  let us denote by $\tir{V}(x)$ the minimal subset of $V(x)$ such that
$$
\partial_V v_h(x) = \{ \, p \in \R^d, p \cdot (h e) \leq v_h(x+he) - v_h(x), \forall e \in \tir{V}(x)  \, \}. 
$$
Thus, if $e \in \tir{V}(x)$, $L_e \neq 0$.

\begin{rem} \label{non-zero-diag}
If $\partial \omega_V(R,v,\{\, x^i \,\}/ \partial v(x^i) =0$, then as $R>0$, all facets of the polygon $Q(\lambda)=\partial_V v_h(x^i)$ have Lebesgue measure 0, and hence $\omega_V(R,v,\{\, x^i \,\})=0$.
\end{rem}


\subsection{Invertibility of the Jacobian of $G$ for $w\neq 0$}

Let us first recall some results on matrices with the structure of the Jacobian of $G$.

The directed graph of a $N \times N$ matrix $A$ is the graph with set of vertices $I=\{ \,1,\ldots,N \, \}$ and edges defined as follows: there exists an edge from $i$ to $j$ if and only if $A_{ij} \neq 0$.

The directed graph of $A$ is said to be connected if there is a path, or sequence of edges, between any two vertices in $\{ \, 1,\ldots, N \, \}$. It is known that if the  directed graph of $A$ is connected, the kernel of $A$ is one dimensional \cite{caughman2006kernels}. 

\begin{thm} \label{det0}
For $w=0$ the $N \times N$ matrix $B=G'(v_h)$ has the following structure: 
$$
B=\begin{pmatrix}
B^{11} & B^{12} \\
B^{21} & B^{22}
\end{pmatrix},
$$
where $B^{11}$ is a $M\times M$ matrix, $B^{12}$ is a $M \times (N-M)$ zero matrix, $B_{21}$ is a $(N-M) \times M$ matrix with entries 0 or 1 with each row containing at most one non zero entry and finally the $(N-M) \times (N-M)$ matrix $B^{22}$ is the negative of the identity matrix.

At each $v$ in $\mathbb{U}_{ 0}$ we have $B_{i i} < 0$, $B_{i j} \geq 0, j \neq i$ and $|B_{ii}| = \sum_{j \neq i} |B_{ij}|$, $i,j=1,\ldots,N$. 

\end{thm}

\begin{proof}
Note that the entries of $B_{ij}, i=1,\ldots,M, j=1,\ldots,N$ are computed from \eqref{j1}--\eqref{j5}. For $i=M+1,\ldots,N$, $G_i(v_h)$ depends only on $v_h(x^i)$ and $v_h(y), y = \Gamma(x^i)$, with $\partial G_i(v)/(\partial v(x^i)) = -1$ and $\partial G_i(v)/(\partial v(y)) = 1$. This also shows that $B_{i i} \leq 0$ and $B_{i j} \geq 0, j \neq i$, $i,j=1,\ldots,N$ with $|B_{ii}| = \sum_{j \neq i} |B_{ij}|$, $i,j=1,\ldots,N$.

By Remark \ref{non-zero-diag}, $B_{ii} \neq 0$ for all $i$ since $\omega_V(R,v_h,\{\, x \,\})>0$ for all $x \in \Omega_h$ when $v_h$ in $\mathbb{U}_{ 0}$. 
Thus $B_{i i} < 0$. 
\end{proof}

\begin{lemma} \label{to-p-invert}
Let $t \in \R^J$ such that $t_i=1, i =1,\ldots,J$ and let $s \in \R^M$ such that $s_1=1$, $s_i=0, i \neq 1$. 
If $w \neq 0$ and $H$ is a $J\times M$ matrix with a one dimensional kernel spanned by $t$, then $H + w \, ts^T$ has a trivial kernel.
\end{lemma}

\begin{proof}
Let $x\in \R^M$ and put $x=\alpha t + y$ for $\alpha \in \R$ and $t$ in the range of $H$. We have
\begin{align*}
(H + w \, t s^T) x = (H + w \, t s^T) (\alpha t + y) = H y + w \, (\alpha s^T t + s^T y ) t, 
\end{align*}
where we used $H t=0$. We note that $\alpha s^T t + s^T y$ is a scalar. If  $(H + w \, t s^T) x =0$, then 
$$H y = -  w \, (\alpha s^T t + s^T y ) t,
$$ 
is both in the kernel and range of $H$. Thus $H y=0$ and then $y$ is both in the kernel and range of $H$. We conclude that $y=0$.  This gives by the above equation $w \, (\alpha s^T t )=0$.  Since $w\neq 0$ and $s^T t=1$, we obtain $\alpha=0$. It follows that $x=0$ and $H + w \, t s^T$ has a trivial kernel. 
\end{proof}

\begin{thm} \label{det}
At each $v$ in $\mathbb{U}_{ 0}$ the matrix $G'(v) $ is invertible for $w  \neq 0$. 
\end{thm}

\begin{proof}
 Put $A=G'(v)$. The $N\times N$ matrix $A$ has the following structure: 
$$
A=\begin{pmatrix}
A^{11} & A^{12} \\
A^{21} & A^{22}
\end{pmatrix},
$$
where $A^{11}$ is a $M\times M$  matrix, $A^{12}$ is a $M \times (N-M)$ zero matrix, $A^{21}$ is a $(N-M) \times M$ matrix with entries 0 or 1 with each row containing at most one non zero entry and finally the $(N-M) \times (N-M)$ matrix $A^{22}$ is the negative of the identity matrix $I^d$. Since $A^{12}$ is a zero matrix and $A^{22}=-I^d$, we have
\begin{equation*} 
\det(A) = \pm \det(A^{11}),
\end{equation*}
by expanding the determinant along the last $M-N$ columns of $A$.

We have $A^{11} = B^{11} + w \, t s^T$ where the matrix $B$ is defined in Theorem \ref{det0} and the vectors $t \in \R^M$ and $s$ are defined in Lemma \ref{to-p-invert}. 

Next, let $I_1, \ldots, I_k$ be a partition of $\{\,1,\ldots,M \, \}$ such that for all $i, j$ in $I_l, l=1,\ldots,k$, there is path from row $i$ to row $j$ of $A^{11}$. Let $H_{l}, l=1,\ldots,k$ be the matrix which rows are the corresponding rows of $A^{11}$ indexed by $I_l$. Put $J_l= \# I_l$ and let $t_l \in \R^{J_l}$ have all entries equal to 1. Upon rearranging the rows of $A^{11}$, we may write for $x\in \R^m$,
$x=[x_1,\ldots,x_k]$ with $x_l \in \R^{J_l}$ for all $l$. 

We have by construction $H_l t_l=0$ for all $l$ and the direct graph of $H_l$ is connected. Thus $H_l$ has a one dimensional kernel spanned by $t_l$. Assume now that $A^{11} x=0$ for $x \in \R^M$. We have $(H_l+ w  \, t_l s^T) x_l=0$ for all $l$. By Lemma \ref{to-p-invert}, $x_l=0$ and thus $x=0$. We conclude that $A^{11}$ is invertible. This completes the proof. 
\end{proof}

\subsection{Proof of Theorem  \ref{main0} }

\begin{thm} \label{proper}
The mapping $G: \mathbb{U}_{0} \to \R^N$ is proper for $w \neq 0$.  
\end{thm}

\begin{proof} 

Let $K$ be a compact subset of $\R^N$. Since $G$ is continuous by Lemma \ref{C1-continuous}, $G^{-1}(K)$ is closed. Because $K$ is bounded, there exists a constant $C$ such that for $v \in G^{-1}(K)$ we have $ || G(v)|| \leq C$. Since for $i=1,\ldots, M$, 
$G_i(v) = \omega_V(R, v ,\{\, x^i \,\})$ $- \int_{C_{x^i}} f(t) dt - w \, v(x^1) $ we obtain
$$
|w| \, | v(x^1) | \leq ||G(v)|| +  \omega_V(R, v ,\{\, x^i \,\}) +  \int_{C_{x^i}} f(t) dt \leq  ||G(v)|| +  \int_{\Omega^*} R(p) dp + \int_{\Omega} f(t) dt.
$$
It follows that for $v \in G^{-1}(K)$, we have $| v(x^1) | \leq C/|w|$ for a constant $C$. 
Using Lemma \ref{pre-lip-lem} we obtain for $x \in \Omega_h$ and $v \in \mathbb{U}_{0}$,
$|v(x)| \leq |v(x^1) | + |v(x) - v(x^1)| \leq C$. This shows that $G^{-1}(K)$ is a bounded subset of $ \mathbb{U}_{0}$. We concluded that $G^{-1}(K)$ is compact. 
\end{proof}

We can now give the proof of Theorem \ref{main0}. 

\begin{proof}[Proof of Theorem \ref{main0}]
The result follows immediately from Theorems \ref{damped-th},  \ref{C1-continuous}, 
\ref{det} and \ref{proper}. 
\end{proof}

\subsection{Uniqueness of solutions}

Theorem \ref{main0} establishes the existence of a solution to \eqref{m2d3} in $\tir{\mathbb{U}}$. If the solution lies  in $\mathbb{U}_0$, by Hadamard’s global inverse function theorem, the mapping $G$ is a homeomorphism and hence the solution is unique. 

We now give a sufficient condition for a solution to be in $\mathbb{U}_0$. Assume that $f>0$ in $\Omega$ and $f \in C(\tir{\Omega})$, we show in section \ref{limit} below that $w \, u_h (x^1_h) \to 0$ as $h \to 0$ and $\kappa \to \infty$, where $\kappa$ measures the size of the stencil $V$. For $h$ sufficiently small, we have $w \, u_h (x^1_h) + \int_{C_x} f(t) dt >0$ for all $x \in \Omega_h$. Then $u_h \in \mathbb{U}_0$ and is the unique solution of  \eqref{m2d3}.

We conclude this section with additional remarks on uniqueness.

\begin{thm} \label{rem-uni}
Assume that $w \neq 0$. Let $u_h$ and $v_h$ be two solutions of \eqref{m2d3}.  Then $u_h(x^1_h)=v_h(x^1_h)$. 
\end{thm}

\begin{proof}
Assume for contradiction  that $u_h(x^1_h) > v_h(x^1_h)$. We distinguish cases based on the sign of $w$. 

If $w<0$, we have  $w u_h(x^1_h) < w v_h(x^1_h)$.
It follows that $ \omega_V(R,u_h,\{\, x \,\}) <  \omega_V(R,v_h,\{\, x \,\})$ for all $x \in \Omega_h$. Let $z \in \Omega_h$ such that $u_h(x) - v_h(x) \geq u_h(z) - v_h(z)$ for all $x \in \Omega_h$. By Lemma \ref{key-for-u} we have $ \omega_V(R, u_h,\{\, z\, \})  \geq  \omega_V(R, v_h,\{\, z\, \}) $. A contradiction. 

If $w>0$, we have  $w u_h(x^1_h) > w v_h(x^1_h)$ and  $ \omega_V(R,u_h,\{\, x \,\}) >  \omega_V(R,v_h,\{\, x \,\})$ for all $x \in \Omega_h$. Now we let $z \in \Omega_h$ such that $v_h(x) - u_h(x) \geq v_h(z) - u_h(z)$ for all $x \in \Omega_h$. This gives 
 $ \omega_V(R, v_h,\{\, z\, \})  \geq  \omega_V(R, u_h,\{\, z\, \}) $ and we get again a contradiction. 
 
 Thus $u_h(x^1_h) \leq v_h(x^1_h)$. Switching the roles of $u_h$ and $v_h$, we obtain $u_h(x^1_h)=v_h(x^1_h)$. 
\end{proof}

By Theorem \ref{rem-uni}, for two solutions $u_h$ and $v_h$ of \eqref{m2d3}, we have $u_h(x^1_h)=v_h(x^1_h)$ and thus 
$\omega_V(R, u_h ,\{\, x \,\}) = \omega_V(R, v_h ,\{\, x \,\})$ for all $x \in \Omega_h$. In particular, as shown in the numerical example in section \ref{num-exp}, if we have
$\omega_V(R, u_h ,\{\, x \,\}) = \omega_{V_{max}}(R, u_h ,\{\, x \,\})$ for all $x \in \Omega_h$, uniqueness is guaranteed, as then $u_h$ is equal to its convex envelope \cite{Awanou-second-sym}. 

\section{Limit equations } \label{limit} 

In this section, we denote by $u_h^w$ the solution of \eqref{m2d3}. We study the convergence of $w u_h^w(x^1_h)$ as $h \to 0$.


\subsection{Weak solutions of \eqref{m2}}

Let $u$ be an Aleksandrov solution of \eqref{m1}. Since $\partial u(\Omega) = \Omega^*$ is bounded, $u$ is Lipschitz continuous on $\Omega$. The proof is analogous to the proof of \cite[Lemma 1.1.6]{Gutierrez2016}. It follows that $u\in C(\tir{\Omega})$. Thus $u(x^1)$ is well defined. 

A convex function $u$ on $\Omega$ is an Aleksandrov solution of \eqref{m2} if for all Borel sets $E\subset \Omega$,  
$\omega(R,u,E) =\int_{E} f(x) dx + w |E| u(x^1)$, and we recall that $|E|$ denotes the Lebesgue measure of $E$.

If $u$ is the unique Aleksandrov solution of \eqref{m1} with $u(x^1)=0$, $u$ is also an  Aleksandrov solution of \eqref{m2}. Conversely, if $u$ is an  Aleksandrov solution of \eqref{m2}, by the compatibility condition \eqref{necessary}, $u$ is the Aleksandrov solution of \eqref{m1} with $u(x^1)=0$. We conclude that  \eqref{m2} has a unique Aleksandrov solution. 

Next, we adapt the notion of $\kappa$-viscosity solution used in  \cite{Awanou-second-sym}. The condition number of a quadratic function $\phi$ is defined to be the condition number of the matrix $M$ such that $\phi(x) = x^T M x, x \in \R^d$. 

We define a function $u$ on $\Omega$ to be $\kappa$-locally convex if for all $x, y \in \Omega$ such that $||y-x|| \leq 1/2 \sqrt{d} \kappa$, we have for $\theta \in (0,1)$
$u(\theta x + (1-\theta) y) \leq \theta u(x) + (1-\theta) u(y)$. We refer to \cite{Mirebeau15,neilan2019monge} for the motivation for the factor $ 1/2 \sqrt{d}$ in proving consistency of discretizations for the Monge-Amp\`ere operator. 

\begin{defn}
A $\kappa$-locally convex function $u$ on $\Omega$ is a $\kappa$-viscosity solution of \eqref{m2}, if for all $x_0 \in \Omega$ and all strictly convex quadratic polynomials $\phi$ with condition number less than $\kappa$, the following holds
\begin{itemize}
\item[-] at each local maximum point $x_0$ of $u-\phi$, $w u(x^1) + f(x_0) \leq R(D \phi (x_0)) \det D^2 \phi(x_0)$
\item[-] at each local minimum point $x_0$ of $u-\phi$, $w u(x^1) + f(x_0) \geq R(D \phi (x_0)) \det D^2 \phi(x_0)$, if $D^2 \phi(x_0) \geq 0$, i.e. $D^2 \phi(x_0)$ has positive eigenvalues.
\end{itemize}
\end{defn}

A viscosity solution of \eqref{m2} is a $\kappa$-viscosity solution for all $\kappa >0$. As with Aleksandrov solutions of \eqref{m1}, for $f>0$ and $f\in C(\tir{\Omega})$, $u$ is a viscosity solution of  \eqref{m2} if and only if $u$ is an Aleksandrov solution of \eqref{m2}, c.f. \cite{Feida2013}. We will treat separately the asymptotic cone condition \eqref{extension-00}. 

\subsection{Converging subsequences}

For $\kappa >0$, define $V_{\kappa}$ to be a mesh independent stencil such that 
$V_{\kappa}$ consists of all vectors $e \in \mathbb{Z}^d \setminus \{ \, 0 \, \}$ with co-prime coordinates and with $|e| \leq 1/2 \sqrt{d} \kappa$. The stencil $V_{\kappa}$ is used to prove consistency of the discretization of the Monge-Amp\`ere operator in \cite{Awanou-second-sym}.

Define
$$
\widetilde{\Omega}_{ext}  = \Omega_h \cup \{ \, x+ h e: x \in \Omega_h, e \in \{
\, r_1,\ldots, r_d \, \} \, \}, 
$$
and we recall that 
$\{\, r_1,\ldots, r_d \, \}$ denotes the canonical basis of $\R^d$. 
We have
 $$
\tir{\Omega} \subset \Conv(\widetilde{\Omega}_{ext}),
 $$
 where $\Conv( \widetilde{\Omega}_{ext} ) $ denotes the convex hull of $\widetilde{\Omega}_{ext}$. 
 
 \begin{defn} \label{def-cv-mesh2}
We say that $u_h$ converges to a function $u$ uniformly on $\tir{\Omega}$ if and only if for each sequence $h_k \to 0$ and for all $\epsilon >0$, there exists $h_{-1} >0$ such that for all $h_k$, $0< h_k < h_{-1}$, we have
$$
\max_{x \in\widetilde{\Omega}_{ext} } |u_{h_k}(x) - u(x)| < \epsilon.
$$
\end{defn}
In the above definition, one can use any extended domain $\widehat{\Omega}_{ext}$ for which $\tir{\Omega} \subset \Conv(\widehat{\Omega}_{ext})$. In \cite{Awanou-second-sym}, we used $\Omega_{ext}$. This allows to prove uniform convergence over $\tir{\Omega}$. We will prove the following theorem

\begin{thm} \label{final-thm-0002} 
Assume that $u_{h,\kappa}$ is $V$-discrete convex and solves \eqref{m2d3} for $V(x)=V_{\kappa}  \cap V_{max}$  for all $x \in \Omega_h$ with 
$x^1_h \to x^1 \in \tir{\Omega}$. There is a subsequence $h_k$ such that  $u_{h_k,\kappa}$ converges uniformly on $\tir{\Omega}$ to a continuous $\kappa$-locally
convex function $v_{\kappa}$.
\end{thm}

Theorem \ref{final-thm-0002} improves on \cite[Theorem 23]{Awanou-second-sym} in that $\Omega$ is not assumed to be a rectangle. Let $\diam(\Omega)$ denote the diameter of $\Omega$.

\begin{thm} \label{boom} Assume that $u_h$ is $V$-discrete convex and $h\leq \diam(\Omega)/2$. 
There exists a constant $C$ which depends only on $\Omega^*$ such that for all $x, y \in \widetilde{\Omega}_{ext}$ we have
$$
|u_h(x) - u_h(y)| \leq C ||x-y|| + C h.
$$
As a consequence, for a triangulation $\mathcal{T}_h$ with set of vertices $\widetilde{\Omega}_{ext}$, the piecewise linear interpolant $I(u_h)$ of $u_h$ also satisfies the above approximate Lipschitz property. 
\end{thm}

\begin{proof}
By Lemma \ref{pre-lip-lem}, for all $x_1, x_2 \in \Omega_h$ we have $|u_h(x_1) - u_h(x_2)| \leq C_1 ||x_1-x_2|| $. 

Let $x_1, x_2 \in \widetilde{\Omega}_{ext} \setminus \Omega_h$. We have $x_i=y_i + h e_i, i=1,2 $ with $y_i \in \Gamma_h$ and $e_i \in \{
\, r_1,\ldots, r_d \, \}$. 

Let $i=1, 2$. We have $u_h(x_i) \leq u_h(y_i) + k_{\Omega^*}(x_i-y_i)$. Therefore, 
$$
u_h(x_i) - u_h(y_i) \leq k_{\Omega^*}(x_i-y_i) = k_{\Omega^*}( h e_i) = h k_{\Omega^*}(  e_i) . 
$$
Thus, if $u_h(x_i) - u_h(y_i) \geq 0$, 
$$
|u_h(x_i) - u_h(y_i)| \leq h | k_{\Omega^*}(  e_i) | \leq C_2 h,
$$ 
where $C_2 = \max \{\, |k_{\Omega^*}(e)|, e \in \{\, r_1,\ldots, r_d \, \} \, \}$. 

On the other hand, if $u_h(x_i) - u_h(y_i) \leq 0$, by $V$-discrete convexity, $u_h(x_i) - 2 u_h(y_i) + u_h(y_i-h e_i) \geq 0$ and thus
$u_h(x_i) - u_h(y_i)  \geq u_h(y_i) - u_h(y_i-h e_i)$. Since $h\leq \diam(\Omega)/2$, $y_i-h e_i \in \Omega_h$ and using again
Lemma \ref{pre-lip-lem}, we obtain
$$
 |u_h(x_i) - u_h(y_i)| = u_h(y_i) - u_h(x_i) \leq u_h(y_i-h e_i) - u_h(y_i) \leq C_1 h ||e_i|| = C_1 h.
$$
We conclude that for $i=1, 2$
$$
 |u_h(x_i) - u_h(y_i)| \leq C_3 h,
$$
where $C_3=\max \{ \, C_1, C_2\, \}$.
We have
\begin{multline*}
|u_h(x_1) - u_h(x_2)| \leq |u_h(x_1) - u_h(y_1)| + |u_h(y_1) - u_h(y_2)| + | u_h(y_2) -  u_h(x_2)| \\
\leq 2 C_3 h + C_1 ||y_1-y_2|| \leq  2 C_3 h + C_1 (||x_1-x_2|| + h ||e_1-e_2||) \\
\leq C_1 ||x_1-x_2|| + 2 (C_1+C_3) h. 
\end{multline*}
Finally, we consider the case $x_1 \in \Omega_h$ and $x_2 \in \widetilde{\Omega}_{ext} \setminus \Omega_h$. Put $x_2 = y_2 + h e_2, y_2 \in \Gamma_h$ and $e_2 \in \{\, r_1,\ldots, r_d \, \}$.  As in the previous case, we obtain $|u_h(x_1) - u_h(x_2)| \leq C_1 ||x_1-x_2|| + (C_1+C_3) h$. 

The second statement of the theorem is proven as in the proof of \cite[Theorem 11]{awanou2019uweakcvg}. 
\end{proof}
We now use an approximate Arzel\`a–Ascoli theorem to obtain compactness in the space of continuous functions.
\begin{defn}  \cite[Appendix A]{calder2015pde}
Let $X \subset \R^d$ be a compact set. We say that a sequence $g_n: X \to \R$ is approximately equicontinuous if for every $\epsilon>0$, there exists $\delta >0$ such that for all $x, y \in X$ and $||x-y||_1 < \delta$, we have $|g_n(x) 
-g_n(y)| < \epsilon+ 1/n$, for all $n$.
\end{defn}

\begin{thm}  \cite[Appendix A]{calder2015pde} \label{ap-arzela}
Let $g_n: X \to \R$ for $X \subset \R^d$ a compact set, be approximately equicontinuous and uniformly bounded. Then there is a subsequence of $g_n$ converging uniformly on $X$ to a continuous function $g: X \to \R$. 
\end{thm}

\begin{proof}[Proof of Theorem \ref{final-thm-0002}]
We have $G(u_{h,\kappa})=0$ where $G$ was defined in \eqref{Gg}. Since $G$ is proper by Theorem \ref{proper}, $u_{h,\kappa}$ is uniformly bounded. Using Theorems \ref{boom} and \ref{ap-arzela} we obtain 
the uniform convergence on  $\tir{\Omega}$ of a subsequence 
$u_{h_k,\kappa}$  to a continuous function $v_{\kappa}$. The $\kappa$-local
convexity of $v_{\kappa}$ is proven as for the convexity of the uniform limit of $V_{max}$-discrete convex functions \cite[Lemma 16]{awanou2019uweakcvg}. 
\end{proof}

\subsection{Equations solved by the limit}

As with \cite[Theorem 22]{Awanou-second-sym}, the limit $v_{\kappa}$ from Theorem \ref{final-thm-0002} is a $\kappa$-viscosity solution of \eqref{m2}. 
It is immediate that the limit of $\kappa$-locaaly convex function as $\kappa \to \infty$ is convex. 
Furthermore, as with  \cite[Theorem 24]{Awanou-second-sym},  $v_{\kappa}$ converges uniformly on 
$\tir{\Omega}$ to a viscosity solution $v$ of \eqref{m2} as $\kappa \to \infty$. 

Under the assumption that  $f>0$ and $f\in C(\tir{\Omega})$, a viscosity solution of \eqref{m2} is also an Aleksandrov solution of \eqref{m2}. %

We note that $\Gamma_h \subset \Omega_h \subset \widetilde{\Omega}_{ext}$, so a solution $u_h$ of \eqref{m2d3} can be extended to $\mathbb{Z}^d_h$ using \eqref{extension}. Consequently, by an argument analogous to that used inPart 2 of \cite[Theorem 20]{Awanou-second-sym}, the limits $v_{\kappa}$ and $v$ can be extended  to $\R^d$ and satisfy the asymptotic cone condition \eqref{extension-00}. It follows that the limit $v$ is the unique Aleksandrov solution $u$ of  \eqref{m2} with asymptotic cone $K_{\Omega^*}$. For the latter $u(x^1)=0$. 

We conclude that as $h_k \to 0$, $w \, u_{h_k}(x^1_{h_k}) \to w \, v_{\kappa}(x^1)$ and as $\kappa \to \infty$, $w v_{\kappa}(x^1) \to 0$. Moreover, the convergence of $u_{h_k,\kappa}$ to $v_{\kappa}$ is uniform in $\kappa$ implying that $u_{h_k,\kappa}$ converges uniformly to $v$ as $h_k \to 0$ and $\kappa \to \infty$. Since the limit $v=u$ is unique, we conclude that $u_{h,\kappa} \to u$  uniformly on $\tir{\Omega}$. So, as $h \to 0$ and $\kappa \to \infty$, $w u_h(x^1_h)\to 0$.


\section{Numerical experiments} \label{num-exp}

Better accuracy is obtained with $v_h(x) = u_h(x) + (u_h(x^1) - u(x^1))$, which is used in our experiments. It is possible to have faster computation times when the origin is an interior point of $\Omega^*$. This can be achieved by adding a linear function to the solution and translating $R$. 

For convenience, for $L(x) = a^* \cdot x$, we put 
$\partial (v+L) (y) = \partial (v+a^* \cdot x) (y)$.
We have for $y \in \R^d$
\begin{equation} \label{trans-a*}
\partial (v+a^* \cdot x) (y) = \partial v(y) + a^*.
\end{equation}
This follows from the following equivalences: $q \in \partial (v+a^* \cdot x) (y)$ if and only if $v(z)+L(z) \geq v(y) +L(y) + q \cdot (z-y), \, \text{for all} \, z \in\R^d$, i.e. $v(z) \geq v(y) + (q-a^*) \cdot (z-y),  \, \text{for all} \, z \in\R^d$, i.e. $q-a^* \in  \partial v(y)$.  

We conclude from \eqref{trans-a*} that 
\begin{equation} \label{trans-a2}
\partial u (\Omega)  = \Omega^* \ \text{if and only if } \partial (u+a^* \cdot x) (\Omega)  =  \Omega^* +a^*.
\end{equation}
Recall  the equation in measures
\begin{align} \label{m1m}
\begin{split}
\omega(R,u,E) &= \int_E f(x) d x \text{ for all Borel sets } E \subset \Omega \\ 
\partial u (\Omega)  &= \Omega^* .
\end{split}
\end{align}
Let us denote by $\omega(R(p-a^*),v+a^* \cdot x,E)$ the $R$-Monge-Amp\`ere measure of $v(x)+a^* \cdot x$ associated with $R(p-a^*)$. We have by \eqref{trans-a*} and a change of variable
\begin{multline*}
\omega(R(p-a^*),v+a^* \cdot x,E) =  \int_{ \partial (v+a^* \cdot x) (E)} R(p-a^*) d p
=  \int_{\partial v (E)+a^*} R(p-a^*) d p \\ =  \int_{\partial v(E)} R(p) d p
=\omega(R,v,E).
\end{multline*}
Therefore, from \eqref{trans-a2}, $u$ solves \eqref{m1m} if and only if $u$ solves
\begin{align} \label{m1mm}
\begin{split}
\omega(R(p-a^*),v+a^* \cdot x,E) &= \int_E f(x) d x \text{ for all Borel sets } E \subset \tir{\Omega} \\ 
 \partial (v+a^* \cdot x) (\Omega)  &=  \Omega^* +a^*.
\end{split}
\end{align}
We note that for \eqref{m1mm}, the analogue of the compatibility condition \eqref{necessary} holds.
Based on \eqref{m1mm} we may assume, without loss of generality, that $\Omega^*$ contains the origin as an interior point. 

The numerical results reported below are for $w<0$. Similar results were obtained for $w>0$.
We take $d=2$, $\Omega=(0,1)^2$, $a=(1/2,1/2)$, $x^1_h=h a$,  $V(x)=-W \cup W$ for all $x \in \Omega_h$ where 
$$W= \{\, (1,0), (0,1), (1,1), (1,-1), (2,1), (-1,2), (1,2), (-2,1) \, \}.$$ The initial guess is a quadratic function $p$ such that $\partial p(\Omega)$ is a rectangle contained in $\Omega^*$.

We give an example in the degenerate case $f \geq 0$. We take $R(x,y)=1$ and the exact solution 
$u(x,y)=\max\{ \, 0, |x-1/2|-1/8, |y-1/2|-1/8  \, \}$. This is a piecewise linear convex function with vertices $a_1=(3/8,5/8)$, 
$a_2=(3/8,3/8)$, $a_3=(5/8,3/8)$ and $a_4=(5/8,5/8)$. Here, $\tir{\Omega^*}$ is the square with vertices 
$(-1,0), (0,-1), (1,0)$ and $(0,1)$
with area 2. We have $|\partial u(a_i)|=1/2, i=1,\ldots,4$. Thus $f(x)=0$ in $\Omega$ except at the points $a_i$ where its value is $1/2$. 

The approximations are not accurate unless all four vertices are included as grid points. For example, if none of the vertices lie on the grid, the problem effectively solved becomes $\det D^2 u(x)=0, |\partial u(\Omega)|=2$, which would not satisfy \eqref{necessary}. For 
$h=3/(8(1/2+37))$, we obtained a maximum error of $1.13 \times 10^{-5}$ and for $h=3/(8(1/2+49))$, a maximum error of $1.63 \times 10^{-5}$. Round-off errors dominate for this singular test case. We note that the exact solution is not strictly convex, while our approximations are strictly discrete convex.  For $h=3/(8(1/2+49))$, the smallest value of $\Delta_{he} u_h(x)$ is $-7 \times 10^{-15}$ and the smallest value of $\omega_V(R,u_h,\{ \, x \, \})$ is $-6 \times 10^{-11}$. We show in Figure \ref{subd-pic2} both the exact solution and the approximation. When we used $w=1$, there were points at which $\Delta_{he} u_h(x)=0$. However this did not hold for $h=3/(8(1/2+37))$. 

\begin{figure}[tbp] 
\begin{center}
\includegraphics[angle=0, height=5cm]{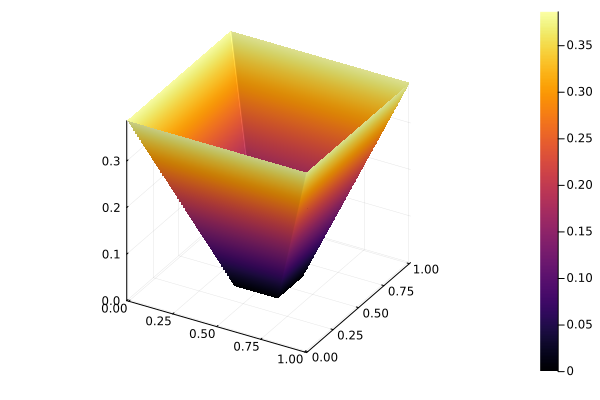}
\includegraphics[angle=0, height=5cm]{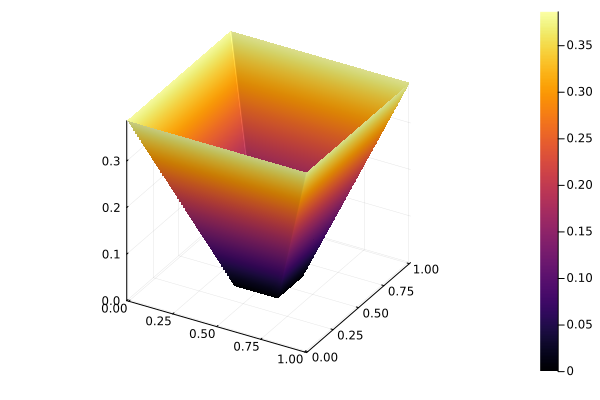}
\end{center}
\caption{Test2: Left: exact solution. Right: approximation for $h=3/(8(1/2+49))$.} \label{subd-pic2}
\end{figure}

\section*{Acknowledgments}

The author was partially supported by NSF grant DMS-1720276. The author would like to thank the Isaac Newton Institute for Mathematical Sciences, Cambridge, for support and hospitality during the programme ''Geometry, compatibility and structure preservation in computational differential equations'' where a significant part of this work was undertaken. This work was partially supported by EPSRC grant no EP/K032208/1. ChatGPT was used for editing and improving the writing of the paper. ChatGPT was used for editing and improving the writing of the paper.

\bibliographystyle{abbrv}  

\end{document}